\documentclass[a4paper, 12pt]{article}
\usepackage{amsmath, amsfonts, amssymb, mathrsfs, enumerate, enumitem, titlesec}
\usepackage[titletoc]{appendix}
\usepackage{amsthm}

\title{Definably Topological Dynamics of $p$-Adic Algebraic Groups}
\date{}

\date{\today}

\author{Jiaqi Bao and Ningyuan Yao}

\setenumerate[1]{itemsep=0pt,partopsep=0pt,parsep=\parskip,topsep=5pt}

\def\Ga{\mathbb{G}_a}
\def\Gm{\mathbb{G}_m}
\def\nn{{\mathcal{N}}}

\def\tp#1{\textnormal{tp}(#1)}
\newcommand\tpo[3][]{\textnormal{tp}^{#1}\left(#2\middle/#3\right)}

\newcommand\dcl[2][]{\textnormal{dcl}^{#1}(#2)}

\newcommand\cl[2][]{\textnormal{cl}^{#1}(#2)}

\def\alg{\textnormal{alg}}

\newcommand\Q{{\mathbb Q}_p}
\newcommand\Z{{\mathbb Z}_p}

\newcommand\z[1][]{\mathbb{Z}_{#1}}
\def\r{\mathbb{R}}
\newcommand\q[1][]{\mathbb{Q}_{#1}}

\def\pM{p^*_{M}}
\def\pM0{p^*_{M_0}}

\def\i{\mathcal{I}}
\def\m{\mathcal{M}}
\def\pcf{p\mathrm{CF}}

\def\Gen{\mathrm{Gen}}
\def\ext{\mathrm{ext}}
\def\j{\mathcal{J}}
\def\M{\mathbb{M}}
\def\K{\mathbb{K}}
\def\N{\mathbb{N}}

\def\pcf{p\textnormal{CF}}
\def\Th{\textnormal{Th}}
\def\Def{\textnormal{Def}}

\def\GL{\textnormal{GL}}
\def\SL{\textnormal{SL}}
\def\TM{\textnormal{Tr}_{{}_M}}
\def\TN{\textnormal{Tr}_{{}_N}}
\def\ext{\textnormal{\lowercase{ext}}}
\def\ia{\i_{_A}^{^M}}
\def\ik{\i_{_K}^{^M}}
\def\jh{\j_{_H}^{^M}}
\def\jb{\i_{_B}^{^M}}
\def\eb{E_B^{^M}}
\def\ea{E_A^{^M}}
\def\id{\text{id}}
\newcommand\stb[2][]{\textnormal{Stab}_{#1}(#2)}

\def\st{\textnormal{st}}
\def\sq{\subseteq}
\def\Ind#1#2{#1\setbox0=\hbox{$#1x$}\kern\wd0\hbox to 0pt{\hss$#1\mid$\hss}
	\lower.9\ht0\hbox to 0pt{\hss$#1\smile$\hss}\kern\wd0}

\def\notind#1#2{#1\setbox0=\hbox{$#1x$}\kern\wd0
	\hbox to 0pt{\mathchardef\nn=12854\hss$#1\nn$\kern1.4\wd0\hss}
	\hbox to 0pt{\hss$#1\mid$\hss}\lower.9\ht0 \hbox to 0pt{\hss$#1\smile$\hss}\kern\wd0}

\def\invlim{\mathop{\lim\limits_{\longleftarrow}}}

\newtheorem{dfn}{Definition}[subsection]
\newtheorem*{Thm}{Theorem}
\newtheorem*{Cor}{Corollary}
\newtheorem*{Conj}{Conjecture}
\newtheorem*{Example}{Example}
\newtheorem{thm}{Theorem}
\newtheorem{fct}[dfn]{Fact}
\newtheorem{lem}[dfn]{Lemma}
\newtheorem{pro}[dfn]{Proposition}
\newtheorem{cor}[dfn]{Corollary}

\begin{document}
	\bibliographystyle{siam}
	
	\maketitle
	
	\begin{abstract}
		We study  the $p$-adic algebraic groups $G$ from the  definable topological-dynamical point of view. We consider the case that $M$ is an arbitrary $p$-adic closed field and $G$  an algebraic group over $\Q$ admitting an Iwasawa decompostion $G=KB$, where $K$ is open and definably compact  over $\Q$, and $B$ is a borel subgroup of $G$ over $\Q$. 
		Our main result is an explicit description of the minimal subflow and Ellis Group of the universal definable $G(M)$-flow $S_G(M^\ext)$. We prove  that the Ellis group of $S_G(M^\ext)$ is isomorphic to the Ellis group of $S_B(M^\ext)$, which is $B/B^0$. 
		
		As applications, we conclude that the Ellis groups corresponding to $\GL(n,M)$ and $\SL(n,M)$ are isomorphic to $(\hat \z \times \Z^*)^n$ and $(\hat \z \times \Z^*)^{n-1}$ respectively, generalizing the main result of Penazzi, Pillay, and Yao in \cite{PPY-sl2qp}.
	\end{abstract}

	\section{Introduction}

	In this paper, we consider the topological dymanics of    algebraic groups over a $p$-adically closed field. The model theoretic approach to topological dynamics was introduced by Newelski  \cite{Newelski-TD-gp-Act}, then
	developed by a number of papers, including \cite{Newelski-Ellis-Semigp},  \cite{Pillay-TD-DG} and \cite{Krupinski-DTD}, and now called definable topological dynamics. Definable topological dynamics studies the action of a group $G$ definable in a structure $M$ on its type space $S_G(M)$ and tries to link the invariants suggested by topological dynamics with model-theoretic invariants. For example, in the case when $Th(M)$ is stable, Newelski proved that the Ellis group of $S_{G,ext}(M)=S_G(M)$ is isomorphic to the  definable Bohr compactification $G^*/{G^*}^{00}$ of G, where $G^*$ is the interpretation of G in a saturated elementary extension \cite{Newelski-TDstable-gp}, which is another formulation of fundamental theorems of stable group theory by replacing the generic types with the Ellis group. Newelski tried to generalize this result to some tame unstable context where generic types may not exist, he conjectured in \cite{Newelski-TD-gp-Act} that such isomorphism holds ture for $NIP$ context. 
	
	A considerable amount of work was motivated by the Ellis group conjecture. The first counterexample of this conjecture is found in \cite{slr}, where the authors showed that Newelski’s conjecture fails in the case of $G=\SL(2,\r)$. G. Jagiella provided a range of counterexamples by extending results  to groups over $\r$ with compact-torsion-free decomposition in \cite{Jagiella-I}, and the second author of this paper showed that their results could be extended to any  elementary extension  $M$ of an $o$-minimal extension of the reals \cite{Yao-I}. Namely, the Ellis groups of $S_{G,ext}(M)$ is isomorphic to the Ellis groups of  $S_G(\r)$. The study of $\SL(2,\q[p])$ in \cite{PPY-sl2qp} provided another counterexample. Kirk showed in \cite{Kirk} that $\SL(2,\mathbb C((t)))$ is also a counterexample. In fact, The main results of \cite{slr, PPY-sl2qp, Kirk} showing that the Ellis group  corresponding to ${\SL}(2,M)$  is isomorphic to $B^*/{B^*}^{00}$,  where $B$ is the Borel subgroup of $\SL(2,-)$, when $M$ is $\r$, $\Q$, or $\mathbb C((t))$. A recent result in \cite{Jagiella-TD-NIP-Field} of G. Jagiella   showed that there is a onto homomorphism from the Ellis groups of $\SL_2(K)$ to $B^*/{B^*}^{00}$ if $K$ has $NIP$.

	As mentioned in \cite{Jagiella-Invs}, one may generalize these results to the case where $G$ has a ``nice'' decomposition in the $p$-adic setting. In this  paper, We provide a way of computing the Ellis groups for $p$-adic algebraic groups that admit a definable ``Iwasawa decomposition''. 
	We finally showed that:
	\begin{Thm}
		Let $G$ be a linear algebraic group defined in $\Q$. Suppose that $G(\Q)$ admits a Iwasawa decomposition $G=KB$ with $B$ a borel subgroup of $G$, definable over $\Q$, and $K$ an open compact subgroup of $G$. Then for any $M\succ \Q$  we have that the Ellis group of  $S_{G,ext}(M)$ is isomorphic to $B^*/{B^*}^{00}=B^*/{B^*}^{0}$.
	\end{Thm}
	
	By \cite{Yao-Tri}, any algebraic group trigonalizable over $\Q$ has a global definable $f$-generic (dfg) type, and  by \cite{OP-plg} any definably compact group over $\Q$ has a global finitly satisfiable generic (fsg) type. So the above decomposition is a kind of ``fsg-dfg'' decomposition in the model-theoritic view when $B$ is split (trigonalizable) over $\Q$.

	
	A split reductive algebraic group $G$ over a local field $F$ admits a Iwasawa decomposition $G=KB$ with $B$ a borel subgroup  trigonalizable over $F$ and $K$ a maximal compact subgroup of $G$. We conclude directly from our main theorem that 
	\begin{Cor}
		If $G(M)$ is a split reductive algebraic group over $\Q$, then the Ellis group of   $S_{G,ext}(M)$ is isomorphic to $(\hat{\mathbb Z}\times \Z^*)^m$ for some $m\in \N$. Particularly, if $G$ is $\GL(n,M)$, then then Ellis group of the  $S_{G,ext}(M)$ is isomorphic to $(\hat{\mathbb Z}\times \Z^*)^n$.
	\end{Cor}
	In \cite{Jagiella-Invs}, Jagiella showed that if $G$ is a definably connected group definable in an $o$-minimal expansion of a real closed field $M$, then the Ellis group of the flow $S_{G,\ext}(M)$ is abstractly isomorphic to a subgroup of a compact Lie group. Based on our result, we conjecture that
	\begin{Conj}
		Let $G$ be a  group definable in $p$-adic closed field $M$.  Ellis group of the flow
		$S_{G,\ext}(M)$ is abstractly isomorphic to a profinite group.
	\end{Conj}

	the paper is organized as follows. In the rest of this introduction we recall some notations, definitions and results, from earlier papers, relevant to our results. 
	
	In section 2.1, we will study the minimal subflow and the Ellis group of a definable group $B=A\rtimes N$,  with $A$ a fsg group and $H$  a dfg group, in the $NIP$ environment.

	In Section 2.2 we will prove some general results for groups definable over $\Q$ admiting compact-dfg decomposition.

	In section 2.3, we prove the main results, on the minimal subflows and Ellis group of the action on $G(M)$ on its type space over $M^\ext$, where $M$ is an arbitrary $p$-adic closed field and $G$ a linear algebraic group admits Iwasawa decomposition, making use of the  results of Section 2.1 and 2.2.

	\subsection{Notations}

	
	We will assume a basic knowledge of model theory. Good references are \cite{Pzt-book} and \cite{Maker-book}. Let $\mathbb T$ be a complete theory with infinite models. Its language is $L$ and $\M$ is the monster model, in which every type over a small subset  $A\subseteq \M$ is realized, where ``small" means $|A|<|\M|$.  $M,N, M', N'$ will denote small elementary submodels of $\M$.  By $x,y,z$ we mean arbitrary $n$-variables and $a,b,c\in \M$  the $n$-tuples in $\M^n$ with $n\in \N$. Every formula is an $L_\M$-formula. For an $L_M$-formula $\phi(x)$, $\phi(M)$ denotes the definable subset of $M^{|x|}$ defined by $\phi$, and a set $X\subseteq M^n$ is  definable  if  there is an $L_M$-formula $\phi(x)$  such that $X=\phi(M)$. If $\bar X\subseteq \M^n$ is definable, defined with parameters from $M$, then $\bar X(M)$ will denote $\bar X\cap M^n$, the realizations from $M$, which is clearly a definable subset of $M^n$. Suppose that $X\subseteq \M^n$ is a definable set, defined with parameters from $M$, then we write $S_X(M)$ for the space of complete types concentrating on $X(M)$. We use $\Def(X(M))$ the denote the boolean algebra of all $M$-definable subset of $X(M)$. We use freely basic notions of model theory such as definable type, heir, coheir, .... The book \cite{Pzt-book} is a possible source. Let $A,B$ be subsets of $\M$, and $p\in S(A)$, by $p\upharpoonright B$ we mean  the restriction of $p$ to $B$ if $A\supseteq B$, and $p|B$ the unique heir of $p$ over $B$ if $B\supseteq A$ with $A$ a model and $p$ definable.

	\subsection{Definable topological dynamics}
	
	Assume that $G$ is a group. By a (point-transitive) $G$-flow we mean a compact Hausdorff space $X$ together with a left action of $G$ on $X$ by homeomorphism that contains a dense orbit. A set $Y$ of $X$ is called a subflow if it is a closed subspace of $X$ which is closed under the action of $G$. A subflow flow is minimal if it has no proper subflows.  The minimal subflows  are “dynamically indecomposable” and considered to be the most fundamental $G$-flows. It is easy to see that $Y\sq X$ is minimal iff $Y=\cl{G\cdot y}$ for each $y\in Y$. For each $g\in G$, we consider the homeomorphism $\pi_g:\ X\rightarrow X$ induced by the group action. Let $X^X$ be the collection of all maps from $X$ to itself, equiped with the product topology, which is a compact Hausdorff space by Tychonoff’s theorem. Let $E(X)$ be the closure of the set $\{\pi_g|\ g\in G\}$ in $X^X$. Then $E(X)$ together with the operation $*$ of the function composition is a semigroup, and the group action of $G$ on $E(X)$ given by $g\cdot x=\pi_g*x$ makes $E(X)$ a $G$-flow. It is natural isomorphism to its own Ellis semigroup. For every $x\in X$ the closure of its $G$-orbit is exactly $E(X)(x) =\{f(x) :f\in E(X)\}$. 
	
	Every minimal subflow of $E(X)$ is a minimal left ideal of the semigroup $E(X)$, and homeomorphic to each other as $G$-flows. We sometimes use the phrase “minimal subflow of $E(X)$” to denote the homeomorphism class of minimal subflows of $E(X)$. A minimal subflow $I$ is the closure of the $G$-orbit of every $p\in I$, hence is $E(X)*p$. We call $u\in I$ an idempotent if $u*u=u$. We denote the collection of all idempotents of $I$ by $J(I)$. For any $u\in J(I)$, $(u*I,\ *)$ is a group with $u$ as its identity. We have $I$ is a disjoint union of $u*I$'s with $u\in J(I)$. All those groups are isomorphic to each other, even for different minimal left ideals. We call these groups the ideal groups and call their isomorphism class the Ellis group of the flow $X$. For more details, readers need to see Refs.\cite{mfate,Ellis}.
	
	Now we consider the topological dynamics in the model-theoretic context. Let $M$ be an $L$-structure.  Take a saturated elementary extension $\M$ of $M$. If $U\sq M^n$ is definable, by an externally definable subset $X$ of $U$ we mean a subset of $U$ of the form $Y\cap U$ with $Y$ an $\M$-definable subset of $\M^n$. By $X\sq_\ext U$ we mean $X$ is an externally definable subset of $U$. We write $\Def^\ext(U)$ for the the boolean algebra of  all externally definable subset of $U$, and $S_{U,\ext}(M)$ the space of all ultrafilters of $\Def^\ext(U)$. In model theory, we consider a definable group $G\sq M^n$ acting on its type space $S_{G}(M)$. Clearly $S_G(M)$ is a $G$-flow. By \cite{Newelski-TD-gp-Act}, the Ellis semigroup of $S_G(M)$ is $S_{G,\ext}(M)$, and the semigroup operation of $S_{G,\ext}(M)$ can be explicitly described. We call $S_{G,\ext}(M)$ the universial definable flow of $G$ over $M$.
	
	Let $M^\ext$ be an expansion of $M$ by adding  predicates for all externally definable subsets of $M^n$ with $n\in \N^+$, and $L^\ext_M$, the  associated language of $M^\ext$, is a nature expansion of the language $L$. If $\Th(M)$ has $NIP$ (see \cite{NIP-book} for the details of $NIP$), then $\Th(M^\ext)$ also has $NIP$,  admits quantifier elimination, and all types over $M^\ext$ are definable \cite{Shelah-NIP}. So we can identify  $S_{ext}(M)$ with $S(M^\ext)$ in $NIP$ context.  Let $S_{M, \text{fs}}(\M)$ be the space of global types which is finitely satisfiable in $M$, then the trace of $p$ in $M$, denoted by  $\TM(p)=\{\phi(M)|\ \phi(x)\in p\}$ is in $S_{\ext}(M)$, and it is easy to see that $p\mapsto \TM(p)$ is a homeomorphim between $S_{M, \text{fs}}(\M)$ and $S_{\ext}(M)$. In $NIP$ theories,  replacing $S_{ext}(M)$ by $S_G(M^\ext)$, we use $p\mapsto \TM(p)$ to denote the homeomorphism from $S_{M, \text{fs}}(\M)$ to $S(M^\ext)$, and $q\mapsto q^\M$ to denote the inverse map.
	
	We assume $NIP$ throughout this paper. Now we use the notation $S_G(M^\ext)$ instead of $S_{G,\ext}(M)$. The semigroup operatopn of $S_G(M^\ext)$ is defined as follows: For any $p,q\in S_G(M^\ext)$, $p*q=\{U\sq_\ext G|\ \{g\in M|\ g^{-1}U\in q\}\in p\}$ 
	
	Note that every type over $M^\ext$ is definable, and thus has a unique heir. By \cite{Newelski-TD-gp-Act}, $p*q$ can also be computered as follows: let $a\models p$ and $b\models q|(M^\ext,a)$, then $p*q=\tp{ab/M^\ext}$.
	

	\subsection{$NIP$, definable amenablity, and connected components}
	Let $G=G(\M^n)$ be a definable group. Recall that a type-definable over $A$ subgroup $H\sq G$ is a type-definable subset of $G$ over $A$, and also a subgroup of $G$. We say that $H$ has bounded index if $|G/H|<2^{|A|+|T|}$. If $\M$ has $NIP$, then there is a smallest type-definable subgroup  of bounded index (see \cite{Shelah-min-bounded-index}), we call it the type-definable connected component of $G$, and denote it by $G^{00}$. We call the intersection of all $\M$-definable subgroups of $G$ of finite index the definable connected component, and denote it by $G^0$. Clearly, both $G^{00}$ and $G^0$ are normal subgroups of $G$ and $G^{00}\leq G^0$.  Note that by \cite{CPS-Ext-Def-in-NIP}, $G^{00}$ is the same whether computed in $T$ or in $\Th(M^\ext)$ if $T$ has $NIP$.

	In \cite{Newelski-TD-gp-Act}, Newelski conjectured that $G/G^{00}$ is isomorphic to the Ellis group of $G$ in $NIP$ theories. Chernikov and Simon showed that the conjecture holds when $G$ is definably amenable and $NIP$. Briefly, a group is definably amenable if it admits a global (left) $G$-invariant Keilser measure, where a global Keisker on $G$ is a finitely additive probabilistic measure on the  algebra of all $\M$-definable subsets of $G$.

	We now recall the stability-theoretic notion of dividing: A type $p(x)\in S(B)$ divides over a set $A\subseteq B$ if there is a formula $\phi(x,b)\in p$ and infinite $A$-indiscernible sequence $b_0,b_1,b_2,...$ such that $\{\phi(x, b_i) : i<\omega\}$  is inconsistent. 
	
	A nice result of \cite{CS-Definably-Amenable-NIP-Groups} showing that:
	\begin{fct}
		$G$ is definably amenable iff there exists $p\in S_G(\M)$ such that for every $g\in G =G(\M)$, $gp$ does not divide over $M$. Following the notation of \cite{CS-Definably-Amenable-NIP-Groups} we call a type $p$ as in the right hand side a (global) strongly $f$-generic, over $M$, type of $G$.
	\end{fct}

	Given a definable subset $X$ of $G$, we define $X$ to be
	$f$-generic if for some/any model $M$ over which $X$ is defined any left translate $gX$ of $X$
	does not divide over $M$. As the notation
	suggests, the property does not depend on the model $M$ chosen. Call a complete type $p$
	(over some set of parameters) $f$-generic iff every formula in $p$ is $f$-generic.  In \cite{CS-Definably-Amenable-NIP-Groups}, the authors showed that in $NIP$ theories:

	\begin{fct}\label{definably-amenable}
		If $\M$ has the $NIP$ and  $G\sq \M^n$ is a $A$-definable group, then  $G$ is definably amenable iff it  admits a global  $f$-generic type. Moreover, when $G$ is definably amenable, we have:
		\begin{enumerate} [label=(\roman*)]
			\item $p\in S_G(\M)$ is $f$-generic if and only if it is $G^{00}$-invariant;
			\item A type-definable subgroup $H$ fixing a global $f$-generic type is exactly $G^{00}$;
			\item Any global strongly $f$-generic type is $f$-generic;
			\item For any $M\prec\M$ containing $A$, if $E_G^\ext\subseteq S_G(M^\ext)$ is an Ellis group, then the map $\sigma: E_G^\ext\rightarrow G/G^{00}$ defined by $p\mapsto p/G^{00}$ is an isomorphism.
		\end{enumerate}
	\end{fct}

	Among the strongly $f$-generics, there are two extreme case: 
	
	\begin{enumerate}
		\item There is a small submodel $M$ such that every left $G$-translate of $p\in S_G(\M)$ is finitely satisfiable in $M$, we call such  types the fsg (finitely satisfiable generic);
		\item There is a small submodel $M$ such that every left $G$-translate of $p\in S_G(\M)$ is definable over $M$, we call such types the dfg (definable $f$-generic).
	\end{enumerate}
	Clearly, both fsg  and  dfg  groups are definably amenable. We now discuss these two cases. Let $\stb[l]{p}$ denotes the stabilizer of $p$ with respect to the left group action, and $\stb[r]{p}$ the stabilizer of $p$ with respect to the right group action. By  \cite{gmn} we have:


	\begin{fct}\label{fsg}
		Let $G$ be an $A$-definable fsg group witnessed by a global type $p$ and a small model $M$. Then:
		\begin{enumerate}[label=(\roman*)]
			\item Any left (right) translate of $p$ is a global generic type and is finitely satisfiable in any small model $N\subseteq A$.
			\item $G^{00} = \stb[l]{p} = \stb[r]{p}$.
			\item Let $q$ be a global generic type, then $q$ is finitely satisfiable in any small model $N\subseteq A$ hence can be considered as a type in $S_G({N}^\ext)$. Moreover $q$ is a generic type in the $G(N)$-flow $S_G({N}^\ext)$.
			\item If $A\sq N$, then $G(N)$-flow $S_G(N^\ext)$ has a unique minimal subflow $ \textnormal{Gen}(G(N))$, which where  is the space of all generic types in $S_G(N^\ext)$.
			\item For any $A\sq N,N'\prec\M$,  $\Gen(G(N))$ is homeomorphic to $\Gen(G(N'))$ via $p\mapsto \TN(p^\M)$ (So we simply call it the generic type space of $G$, and denote it by $\Gen(G)$).
			
			\item $\i=\Gen(G)$ is a two-sided ideal.
		\end{enumerate}
	\end{fct}

	\begin{lem}\label{fsgms}
		Suppose that $G\sq \M^n$ is a group definable over $M$ admitting fsg and $\i\sq S_G(M^\ext)$ is the minimal subflow. Then
		\begin{enumerate}[label=(\roman*)]
			
			\item\label{fsgms.i} For any $q \in \i$, the Ellis group contains $q$ is $q * \i = q * S_G(M^\ext)$.
			\item\label{fsgms.ii} For any $q \in \i$ and $\tp{a/M^\ext},\tp{b/M^\ext}\in S_G(M^\ext)$, we have that $q*\tp{a/M^\ext}=q*\tp{b/M^\ext}$ iff $a/G^{00}=b/G^{00}$.
			\item \label{fsgms.iii} For each $q\in \i$ and $r\in S_G(M^\ext)$, there is $s\in q*\i$ such that $q=s*r$.
		\end{enumerate}
	\end{lem}
	\begin{proof}
		\ref{fsgms.i}: let $u \in J(\i)$ be such that $q \in u * \i$. Since   $u * \i$ is a group, we have $u * \i = q * u * \i = q * \i$ is  the Ellis group containing $q$. Clearly, we have $q * \i \subseteq q * S_G(M^\ext) $.  Since $\i$ is a two-sided ideal, $u * S_G(M^\ext) \subseteq \i$, thus 
		\[q *S_G(M^\ext)=q * u * S_G(M^\ext) \subseteq q * \i.\] 
		So $q * \i = q * S_G(M^\ext)$
		as required. 
		
		\ref{fsgms.ii}: Since $q * S_G(M^\ext)$ is the Ellis group generated by $q$, we see that (ii) can be concluded directly from  Fact \ref{definably-amenable}(iv).
		
		\ref{fsgms.iii}: Suppose that $r\vdash k/G^{00}$ and $q\vdash t/G^{00}$. By Fact \ref{definably-amenable} there is $s\in q*\i$ such that $s\vdash (tk^{-1}/G^{00})$. Now $s*r\in s*\i=q*\i$ and $(s*r)/K^{00}=q/K^{00}$, so $s*r=q$ as required.
	\end{proof}

	We now discuss the dfg groups.
	
	\begin{fct}\label{dfg}\cite{P-Y-On-minimal-flows}
		Let $B\sq \M^m$ be a group definable over $M$, and $p\in S_B(\M)$ is a global $f$-generic type. If $p$ is definable over $M$, then  
		\begin{enumerate}[label=(\roman*)]
			\item Every left $G$-translate of $p$ is definable over $M$;
			\item $G^{00}=G^0=\stb[l]{p}$;
			\item $G\cdot p$ is closed, and hence a minimal subflow of $S_G(\M)$.
		\end{enumerate}
	\end{fct}

	\begin{fct}\label{B(Mext)-has-dfg}\cite{CPS-Ext-Def-in-NIP}
		Suppose that  $G\sq \M^n$ is a dfg group definable over $M$. Let $\j\sq S_G(M^\ext)$ be a minimal subflow. Then
		\begin{enumerate}[label=(\roman*)]
			\item  $G$ has a global dfg type with respect to $M^\ext$ in $\Th (M^\ext)$.
			\item  Let $N^*$ be a extension of $M^\ext$ and  $p\in \j$, then the unique heir $\bar p\in S_G(N^*)$  of $p$ is an $f$-generic type. Moreover any $G(N^*)$-translate of $\bar p$ is an global heir of some $q\in S_G(M^\ext)*p$.
		\end{enumerate}
	\end{fct}

	\begin{lem}\label{Mini-flow-of-B}
		Suppose that  $G\sq \M^n$ is a dfg group definable over $M$. Let $\j\sq S_G(M^\ext)$ be a minimal subflow. Then $\j$ is an Ellis subgroup of $S_B(M^\ext)$.
	\end{lem}
	\begin{proof}
		Let $\M^*\succ M^\ext$ be a saturated extension. Let $p\in \j$  and $\bar p\in S_G(\M^*)$ the unique heir of $p$ over $\M^*$. By Fact \ref{B(Mext)-has-dfg}, the $G$-orbit $G\cdot \bar p$ is  homeomorphic to $\j$ via the map $\bar q\mapsto \bar q\upharpoonright M^\ext$. Since $\bar p$ is $G^{00}$-invariant, $G\cdot \bar p$ is  isomorphic to $G/G^{00}$ via $\bar q \mapsto \bar q/G^{00}$. So $q\mapsto q/G^{00}$ is also a ismomorphism from $\j$ to $G/G^{00}$. By Fact \ref{definably-amenable}, we see that $\j$ is an Ellis group.
	\end{proof}
	
	One could conclude directly from the above Fact that 
	\begin{cor}\label{structure of J}
		If $G$ has dfg, then for any $q_1,q_2\in S_G(M^\ext)$ and $p\in \j_M^\ext$, $q_1*p=q_2*p$ iff $q_1/G^0=q_2/G^0$.
	\end{cor}

	\subsection{Groups definable in $(\Q,+,\times, 0,1)$}
	We first give our notations for $p$-adics. By ``the $p$-adics'', we mean the field $\Q$. $M_0$ denotes the structure $(\Q,+,\times, 0,1)$, $\Q^*=\Q\backslash \{0\}$ is the multiplicative group. $\mathbb Z$ is the ordered additive group of integers, the value group
	of $\Q$. The group homomorphism $\nu: \Q^*\longrightarrow \mathbb Z$ is the valuation map.  $\M$ denotes a very saturated elementary extension $(\K, +, \times, 0, 1)$ of $M_0$.  Similarly, $\K^*=\Q\backslash \{0\}$ is the multiplicative group. We sometimes write $\Q$ for $M_0$ and $\K$ for $\M$.
	
	For convenience, we use $\Ga$ and $\Gm$ denote the additive group and multiplicative group of field $\M$ respectively.  So $\Ga(M_0)$   and $\Gm(M_0)$ (or $\Gm(\Q)$) are $(\Q,+)$ and $(\Q^*,\times)$ respectively.
	
	We will be referring a lot to the comprehensive survey \cite{Luc-Belair} for the basic model theory of the $p$-adics. A key point is Macintyre's theorem \cite{Macintyre} that $\Th(\Q,+,\times, 0,1)$ has quantifier elimination in the language of rings $L_{ring}$ together with new predicates $P_n(x)$ for the $n$-th powers for each $n\in \N^+$. Moreover, for any polynormals $f,g\in \Q[\bar x]$, the relation $v(f(\bar x))\leq v(g(\bar x))$ is quantifier-free definable in the Macintyre's language $L_{ring}\cup\{P_n|\ n\in \N^+\}$, in particular it is definable in the language of rings. (See Section 3.2 of \cite{Luc-Belair}.) By \cite{Delon-stbebd-in-Qp}, every type over $\Q$ is definable. A \emph{$p$-adically closed field} is a model of $p\mathrm{CF} := \Th(\Q)$, which has $NIP$  (see \cite{Luc-Belair} for details).  The theory $p$CF also has definable Skolem functions \cite{Dries-defskl}. 
	
	The $p$-adic field $\Q$ is a locally compact topological field, with  basis given by the sets
	\[{\cal B}(a,n) = \{x\in \Q \mid x \neq a \wedge v(x - a) \geq n\}\]
	for $a\in \Q$ and $n\in \mathbb Z$. The valuation ring $\z[p]$ is
	compact. The topology is given by a definable family of a definable sets, so it
	extends to any $p$-adically closed field $M$, making $M$ be a
	topological field (usually not locally compact).

	For any $X\sq \Q^n$, the ``topological dimension", denoted by $\dim(X)$, is the greatest $k\leq n$ such that the image of $X $ under some projection from $M_0^n$ to $M_0^k$ contains an open subset of $\Q^k$. On the other side, as model-theoretic algebraic closure coincides with field-theoretic algebraic closure  (\cite{HP-gp-local}, Proposition 2.11), we see that for any model  $M$ of $\pcf$ the algebraic closure satisfies exchange (so gives a so-called pregeometry on $M$) and there is a finite bound on the sizes of finite sets in uniformly definable families. If $a$ is a finite tuple from $M\models \pcf$ and $B$ a subset of $M$ then the algebraic dimension of $a$ over $B$, denoted by $\dim(a/B)$, is the size of a maximal subtuple of $a$ which is algebraically independent over $B$. 
	
	When  $X\sq \Q^n$  is definable, the algebraic dimension of $X$, denoted by $\alg$-$\dim(X)$, is the maximal $\dim(a/B)$ such that $a\in X(\M)$ and $B$ contains the parameters over which  $X$ is defined. It is important to know that when  $X\sq \Q^n$  is definable, then its algebraic-dimension  coincides with its ``topological dimension'', namely $\dim(X)=\alg$-$\dim(X)$. As a conclusion, for any definable $X\sq \Q^n$, $\dim(X)$ is exactly the algebraic geometric dimension of its Zariski closure.
	
	By a definable manifold $X\sq \Q^n$ over a subset $A\sq \Q$, we mean a  topological space $X$ with a covering by finitely many open subsets $U_1,...,U_m$, and homeomorphisms  of $U_i$ with some definable open $V_i\sq  \Q^n$ for $i=1,...,m$, such that the transition maps are $A$-definable and  continuous. If the transition maps are $C^k$, then we call $X$ a definable $C^k$ manifold over $\Q$ of dimension $n$.   A definable group $G\sq \Q^n$ can be equipped uniquely  with the structure of a definable manifold over K such that the group operation is $C^{\infty}$ (see \cite{Pillay-fields-definable-over-Qp} and \cite{OP-plg}). The facts described above work for any $M\models \pcf$.

	By \cite{OP-plg}, a group $K\sq \M^l$ definable over $M_0$ has fsg iff it is definably compact over $M_0$. The type-definable connected component $K^{00}$ coincides with its definable connected component $K^{0}$, which is also the kernel of the standard part map $\st : K\rightarrow K(M_0)$. Namely, $K^0$ is the set of infinitesimals of $K$ over $M_0$.

	By \cite{PY-dfg}, a group $H\sq \M^k$ definable over $M_0$ has dfg iff there is a  trigonalizable algebraic group $A$ over $M_0$ and a definable homomorphic $f: H\rightarrow A$ such that both $ker(f)$ and $A/im(f)$ are finite. In particular, any  trigonalizable algebraic group over $M_0$ has dfg.

	\section{Main Results}
	
	\subsection{Semi-product of a fsg group and a dfg group}
	
	We now consider a group $B=B(\M)$ definable in a $NIP$ structure $M$, which could be decomposed into  a semi-product $B=A\rtimes H$, where $A$ has fsg, $H$ has dfg, and both of the definable over $M$. Clearly, $B$ is definably amenable. We will study the the minimal subflow and the Ellis group of $S_B(M^\ext)$ in this section.

	\begin{lem}\label{B0=B00}
		$B^{00}=A^{00}\rtimes H^0$.
	\end{lem}
	\begin{proof}
		Since $A/A^{00}$ and $H/H^{0}$ is bounded, we see that $B/(A^{00}\rtimes H^0)$ is bounded. So $B^{00}\sq A^{00}\rtimes H^0$. On the other side,  $A/(B^{00}\cap A)\cong AB^{00}/B^{00}$ and $H/(B^{00}\cap H)\cong HB^{00}/B^{00}$. So  $B^{00}\cap A$ has bounded index in $A$ and $B^{00}\cap H$ has bounded index in $H$, which imples that $A^{00}\sq B^{00}$ and $H^{0}=H^{00}\sq B^{00}$. So $A^{00}\rtimes H^0\sq B^{00}$.
	\end{proof}

	\begin{lem}\label{IA*JH}
		Let $\ia$ be the (unique) minimal subflow of $S_A(M^\ext)$ and $\jh$ be a minimal subflow of $S_H(M^\ext)$, then $\ia*\jh$ is a minimal subflow of $S_B(M^\ext)$. Moreover if $u\in \ia$ and  $v\in\jh$ are idempotents, then  $u*v$  is an idempotent.
	\end{lem}
	\begin{proof}
		Let $N^*\succ M^\ext$ be an $|M|^+$-saturated model.
		Clearly, $u\vdash A^{00}$ and $v\vdash H^0$ since they are idempotents. Let $a_1,a_2,h_1\in B(N^*)$ such that $a_1\models u$, $h_1\models v|(M^\ext,a_1)$, and $a_2\models u|(M^\ext,a_1,h_1)$, let $h_2\models v|N^*$, then
		\[
		u*v*u*v=\tpo{a_1h_1a_2h_2}{M^\ext}=\tpo{a_1a_2h_1^{a_2}h_2}{M^\ext}
		\]
		Since $H^0$ is a normal subgroup of $B$, we see that $h_1^{a_2}\in B^{00}(N^*)$. By Fact \ref{B(Mext)-has-dfg},  $\tpo{h_2}{N^*}$ is $B^{00}(N^*)$-invariant, so $\tpo{h_2}{N^*}=\tpo{h_1^{a_2}h_2}{N^*}$, which implies that 
		\[
		\tpo{a_1a_2h_1^{a_2}h_2}{M^\ext}=\tpo{a_1a_2}{M^\ext}*\tpo{h_1^{a_2}h_2}{M^\ext}=u*u*v=u*v.\]
		So $u*v$ is an idempotent. 
		
		We now show that $\ia*\jh$ is minimal. It suffices to show that $\ia*\jh\sq S_B(M^\ext)*r*s$ for any $r\in \ia$ and $s\in \jh$. Let $\alpha\in \ia$ and $\beta\in \jh$, then there are $a\in A(N^*)$ and $h\in H(N^*)$ such that $\alpha=\tpo{a}{M^\ext}*r$ and $\beta=\tpo{h}{M^\ext}*s$. Let $a'\in A(N^*)$ realize $r|(M^\ext, a, h_0)$ and $h_0\in H(N^*)$  such that $h_0^{a'}/H^0=h/H^0$. Let $h'\models s| N^*$.  Then
		\[
		\tpo{ah_0}{M^\ext}*r*s= \tpo{ah_0a'h'}{M^\ext}=\tpo{aa'h_0^{a'}h'}{M^\ext},
		\]
		and, since $\tp{h_0^{a'}h'/N^*}$ is the heir of $\tp{h_0^{a'}h'/M^\ext}$ by Fact \ref{B(Mext)-has-dfg}, it is easy to see that
		\begin{align*}
			\tpo{aa'h_0^{a'}h'}{M^\ext}&=\tpo{aa'}{M^\ext}*\tpo{h_0^{a'}h'}{M^\ext}\\
			&=\tpo{a}{M^\ext}*\tpo{a'}{M^\ext}*\tpo{h_0^{a'}h'}{M^\ext}\\
			&=\tpo{a}{M^\ext}*\tpo{a'}{M^\ext}*\tpo{hh'}{M^\ext}\\
			&=\alpha*\beta.
		\end{align*}
		This completes the proof.
	\end{proof}

	\begin{lem}\label{Ellis-group-in-B}
		Let $\ia$ and $\jh$ be minimal subflows of $S_A(M^\ext)$ and $S_H(M^\ext)$ respectively. Let $u\in \ia$ and $v\in \jh$ be idempotents. Then $u*\ia*\jh$ is the Ellis group in $S_B(M^\ext)$ generated by $u*v$.
	\end{lem}
	\begin{proof}
		By Lemma \ref{IA*JH}, $\ia*\jh\sq S_B(M^\ext)$ is a miniaml subflow  and $u*v$ is an idempotent.  The Ellis group containing $u*v$ is $u*v*S_B(M^\ext)*u*v$. 
		
		We first show that $u*v*S_B(M^\ext)*u*v\sq u*\ia*\jh$. Let $p\in S_B(M^\ext)$. Take $a_1,a_2,a_3\in A$ and $h_1,h_2,h_3\in H$ such that 
		\begin{align*}
			&a_1\models u,h_1\models v|(M^\ext, a_1),\  a_2h_2\models p|(M^\ext, a_1,h_1),\\
			&a_3\models u|(M^\ext, a_1,h_1,a_2,h_2),\  \text{and}\\
			&h_3\models v|(M^\ext, a_1,h_1,a_2,h_2,a_3).
		\end{align*}
		Then $u*v*p*u*v=\tp{a_1h_1a_2h_2a_3h_3/M^\ext}$.
		Now 
		\[
		a_1h_1a_2h_2a_3h_3=a_1a_2{h_1}^{a_2}h_2a_3h_3=a_1a_2a_3({h_1}^{a_2}h_2)^{a_3}h_3.
		\]
		Let $N^*\prec N^{**}$ be  extensions of $M^\ext$ such that $N^*$ is $|M|^+$-saturated and $N^{**}$ is $|N^*|^+$-saturated. Without lose of generality, we may assume that $a_1\in A(N^{**})$ realizes the  coheir of $u$ over $N^*$ and $h_3$ realizes the  heir of $v$ over $N^{**}$, and $h_1,a_2,h_2,a_3\in B(N^*)$.
		Since $\tp{h_3/\M^*}$ is a definable $f$-generic type and $({h_1}^{a_2}h_2)\in H(N^*)$, $\tp{({h_1}^{a_2}h_2)h_3/N^{**}}$ is the unique heir of $\tp{({h_1}^{a_2}h_2)h_3/M^\ext}$. Similarly, $\tp{a_1/\M^*}$ is a finitely satisfiable generic type, so $\tp{a_1a_2a_3/N^*}$ is the unique coheir of $\tp{a_1a_2a_3/M^\ext}$. We conclude that
		\[
		\tp{a_1a_2a_3({h_1}^{a_2}h_2)^{a_3}h_3/M^\ext}=\tp{a_1a_2a_3/M^\ext}*\tp{({h_1}^{a_2}h_2)^{a_3}h_3/M^\ext}
		\]
		Clearly, 
		\[
		\tp{a_1a_2a_3/M^\ext}=\tp{a_1/M^\ext}*\tp{a_2a_3/M^\ext}\in u*S_A(M^\ext),\]
		and 
		\[
		\tp{({h_1}^{a_2}h_2)^{a_3}h_3/M^\ext}=\tp{({h_1}^{a_2}h_2)^{a_3}/M^\ext}*\tp{h_3/M^\ext}\in S_H(M^\ext)*v.
		\]
		By Fact \ref{fsg}(vi), we have $u*S_A(M^\ext)=u*\ia$. Clearly, $S_H(M^\ext)*v=\jh$. So $u*v*S_B(M^\ext)*u*v\sq u*\ia*\jh$. On the other side, if $u_1,u_2\in u*\ia$ and $v_1,v_2\in \jh$ such that $u_1*v_1/B^{00}=u_2*v_2/B^{00}$, then we conclude that $u_1/A^{00}=u_2/A^{00}$ and $v_1/H^0=v_2/H^0$. Since $\jh$ is an Ellis group by Lemma \ref{Mini-flow-of-B}, we have $u_1=u_2$ and $v_1=v_2$, which implies that $p\mapsto p/B^{00}$ is a bijection from $u*\ia*\jh$ to $B/B^{00}$. So $u*\ia*\jh=u*v*S_B(M^\ext)*u*v$ is the Ellis group generated by $u*v$ by Fact \ref{definably-amenable}(iv).
	\end{proof}
	The above Lemma shows that $\ea*\jh$ is an Ellis group in $S_B(M^\ext)$ when $\ea$ is a Ellis group  of $S_A(M^\ext)$ and $\jh$ a minimal subflow (or Ellis group) of $S_H(M^\ext)$.
	
	From now on, we use notation $\ea$ to denote  a Ellis group in $S_A(M^\ext)$ and $\eb=\ea*\jh$.
	
	\begin{lem}\label{Ellis-gp-in-B-II}
		Let $q\in \ea$ and $p\in \jh$, then $\eb=q*S_B(M^\ext)*p$.
	\end{lem}
	
	\begin{proof}
		Clearly, $\ea=q*S_A(M^\ext)$ and $\jh=S_H(M^\ext)*p$. So $\ea*\jh\sq q*S_B(M^\ext)*p$. 
		
		Conversely, let $r\in S_B(M^\ext)$. Let  $N^{*}\succ M^\ext$ such that  $N^*$ is $|M|^+$-saturated. Take $a,a_0\in A(N^*)$ and $h_0\in H(N^*)$ such that $a\models$ the coheir of $q$ over $(M^\ext, a_0,h_0)$ and $a_0h_0\models r$. Let $h\models p|N^*$ Then $\tp{h/N^*}$ is $f$-generic and definable. So the transition $\tp{h_0h/N^*}$ is definable over $M^\ext$, and thus an heir of $\tp{h_0h/M^\ext}$. We see that
		\[q*r*p=\tp{aa_0h_0h/M^\ext}=\tp{aa_0/M^\ext}*\tp{h_0h/M^\ext}\sq \ea*\jh.\]
		This completes the proof.
	\end{proof}

	\begin{cor}\label{group-structure-in-EB}
		Let $\ea$ and $\jh$ be Ellis groups in $S_A(M^\ext)$ and $S_H(M^\ext)$ respectively. Then for any $r_1,r_2,r^*\in \ea$ and $s_1,s_2,s^*\in \jh$ such that $r_1\vdash a_1/A^{00}$, $r_2\vdash a_2/A^{00}$, $s_1\vdash h_1/H^0$,  $s_2\vdash h_2/H^0$ , $r^*\vdash a_1a_2/A^{00}$, and $s^*\vdash h_1^{a_2}h_2/H^0$. Then
		\[
		(r_1*s_1)*(r_2*s_2)=r^**s^*.
		\]
		
	\end{cor}
	\begin{proof}
		It is easy to see from Lemma \ref{Ellis-group-in-B} and Fact \ref{definably-amenable} that $(r_1*s_1)*(r_2*s_2)\vdash  a_1a_2h_1^{a_2}h_2/B^{00}$. On the other side, if $r^*\vdash a_1a_2/A^{00}$ and $s^*\vdash h_1^{a_2}h_2/H^0$, then $r^**s^*\vdash a_1a_2h_1^{a_2}h_2/B^{00}$. So $(r_1*s_1)*(r_2*s_2)=r^**s^*$ as required.
	\end{proof}

	\subsection{Groups with compact-dfg decomposition}
	We assume in this section that $G=G(\M)$ is a group definable in $\M$, with parameters from $M_0=\Q$, and $G = CH$ is a decomposition of $G$, where $H$ is a $\Q$-definable subgroup of $G$ with dfg, and $C$ a $\Q$-definbale subset of $G$ such that $C(M_0)$ is definably compact, and open in $G(M_0)$. 
	
	
	Since $C$ is open in $G$, the infinitesimals of $C$ over $M_0$,   which is the intersection of all $\Q$-definable open subsets of $C$, denoted by $\mu_C$, coincides with $\mu_G$, the  infinitesimals of $G$ over $M_0$. 
	By the continuity of the group operation, we see that $\mu_G^g=\mu_G$ for all $g\in G(M_0)$. Let $V_G=G(M_0)\mu_G$, then it is the  subgroup of $G$ consisting of all elements have its standard part in $G(M_0)$. It is easy to see that $V_G\leq N_G(\mu_G)=N_G(\mu_C)$.

	For any $N\succ M_0$, we use $G^{0}(N)$ to denote $G^0\cap G(N)$. By $V_G(N)$ we mean set $G(M_0)\mu_G(N)$, which is the  subgroup of $G(N)$ consisting of all elements have its standard part in $G(M_0)$.

	Let $Y$ be an $N$-definable subset of $G$. By $Y(N)/H$ we mean the set $\{g/H(N)|\ g\in Y(N)\}$. Let $X=G/H$, we write $\Def(X(N))$ for the boolean algebra of all sets  of the form $\{Y(N)/H|\ Y \in\Def(G(N)) \}$, and $S_{X}(N)$ is the space of all ultrafilters of $\Def(X(N))$, similarly for $\Def^\ext(X)$ and $S_{X}(N^\ext)$.

	We now consider quotient space $X= G/H$, which admits a quotient topology. Let $\pi$ be the projection from $G$ to $X$, then 
	it is easy to see that $\pi$ could be naturally extended to a onto homomorphism from $S_G(M^\ext)$ to $S_X(M^\ext)$.

	\begin{lem}\label{epsilon-equ}
		Let $g,h\in V_G$ such that $\tp{(g/H)/M_0}=\tp{(h/H)/M_0}$, then there is $\epsilon_1,\epsilon_2\in \mu_G$ such that $\epsilon_1 g H=h H$ and $g\epsilon_2 H=h H$.
	\end{lem}
	\begin{proof}
		For any $g,h\in V_G$, we have $\mu_Gg=g\mu_G$ and $\mu_Gh=h\mu_G$. So it sufficies to show that $g\mu_GH\cap h\mu_GH\neq\emptyset$.
		
		If $g\mu_GH\cap h\mu_GH=\emptyset$, then by compactness there is a $M_0$-definable open subgroup $D$ of $C$ such that $D \supseteq \mu_G$ and $gDH\cap hDH=\emptyset$. Since $\mu_G\sq D$, we see that $gD=\st(g)D$ and $hD=\st(h)D$, and thus $g/H\in \st(g)D/H$ and $h/H\notin \st(g)D/H$. We conclude that  $\tp{(g/H)/M_0}\neq\tp{(h/H)/M_0}$. A contradiction.
	\end{proof}

	\begin{lem}\label{kid}
		Let $p \in S_G(M^\ext)$ be $H^0(M)$-invariant. Then for any $k\in K$ and $h\in H$ such that $kh\models p$, ${H^0(M)}^{k}\subseteq \mu_GH$.
	\end{lem}
	\begin{proof}
		Let $h_0\in H^0(M)$. Since $p$ is $H^0(M)$-invariant, we see that 
		\[
		\tp{h_0kh/M^\ext}=\tp{kh/M^\ext},
		\]
		and hence  $\tp{(h_0k/H)/M_0}=\tp{(k/H)/M_0}$. By Lemma \ref{epsilon-equ}, we have $h_0kH= k\epsilon H$ for some $\epsilon\in \mu_G$. So ${h_0}^{k}\in \mu_GH$ for all $h_0\in H^0(M)$ as required.
	\end{proof}

	\subsection{Minimal subflows and Ellis groups of  groups admitting Iwasawa decomposition}

	In this section we assume  that $L$ is the language of the rings, $M_0=(\Q,+,\times,0,1)$ is the standard model of $\pcf$, $M$ will denote an elementary extension of $M_0$, and $L_M^\ext$ the associated language of $M^\ext$. We assume that the $L_M^\ext$-structure $\M^*$ is a monster model of $\Th(M^\ext)$, and $\M=\M^*\upharpoonright L$ is the reduction of $\M^*$ on $L$. Clearly, $\M\succ M_0$ is a monster model of $\pcf$.

	We now consider the case that $G=G(\M)$ is a  linear algebraic group over $\Q$ admitting a  Iwasawa decomposition $G=KB$, where $B$ is a borel subgroup of $G$, definable over $\Q$, and $K$ is a $\Q$-definable open  subgroup of $G$ such that $K(\Q)$ is compact. 
	
	We can decompose $B$ into a semi-product $B=T\rtimes B_u$, where $B_u(\Q)$ is the  maximal unipotent subgroup of $B(\Q)$ and $T(\Q)$ a torus. a basic fact  is that $N_G(B_u)=N_G(B)=B$ (see \cite{lag}). Moreover, $T$ is an almost direct product of $T_{spl}$ and $T_{an}$, where $T_{spl}(\Q)$ is $\Q$-split, thus is isomorphic to $\Gm^k$ for some $k\in \N$, and $T_{an}(\Q)$ is anisotropic, which is compact \cite{AG-and-NT, Aniso}. 
	
	For simplity, We assume that $B=A\rtimes H$ where $A=T_{an}$ and $H=S\rtimes B_u$. By \cite{Yao-Tri}, $H$ has dfg and $H^{00}=H^0=S^0\rtimes B_u$. Let $C=KA$, it is easy to see that $G=CH$ is a compact-dfg decomposition. By \cite{OP-plg}, both $K$ and $A$ have $fsg$, $K^0=K^{00}=\mu_G$, and $A^0=A^{00}=\mu_G\cap A$. By Lemma \ref{B0=B00}, $B^{00}=A^0\rtimes H^0=B^0$, and $B/B^{0}\cong T/T^0$ is commutative.

	\begin{lem}\label{klem}
		Let $p\in S_G(M^\ext)$ be $H^0(M)$-invariant. Suppose that $g \models p $, then $g= \epsilon b$ for some  $\epsilon \in \mu_G$, $b \in B$ with $\epsilon, b \in \dcl{M,g}$.
	\end{lem}
	
	\begin{proof}
		Let $g = kb'$ for some $k\in K$ and $b'\in B$. Since $\pcf$ has definable Skolem functions (see \cite{Dries-defskl}), we may assume $k,b'\in \dcl{M,g}$. By Lemma \ref{kid}, we have that $H^0(M)^{k}\subseteq \mu_G H$. Note that $B_u(M_0)\sq H^0(M)$. So in particular, we have $B_u(M_0)^{k}\subseteq \mu_G H\cap V_G$. Take a standard part map, we have 
		\[
		\st(B_u(M_0)^{k})=B_u(M_0)^{\st(k)}\subseteq\st(\mu_GH\cap V_G)=H(M_0)\sq B(M_0).\]
		Thus $\st(k)\in N_G(B_u(M_0))=B(M_0)$. Let $\epsilon\in \mu_G$ such that $k=\epsilon  \cdot\st(k)$, then $g=\epsilon\cdot\st(k)b' $. Let $b=\st(k)b'$, then $g = \epsilon b\in \mu_G B$ as required. Clearly $\epsilon, b \in \dcl{M,g}$.
	\end{proof}
	
	Now it is easy to see that 
	
	\begin{cor}\label{exchange}
		Let $\jb$ be a minimal subflow of $S_B(\M^\ext)$, $b \models p_1 \in \jb$, and $g \in G(M)$. Then there exist  $\epsilon \in \mu_G\cap \dcl{M,b,g}$ and  $b' \in B\cap \dcl{M,b,g}$  such that $bg = \epsilon b'\in \mu_GB$.
	\end{cor}
	\begin{proof}
		Since $\tp{h_0g/M^\ext}$ is $B^0(M)$-invariant, thus is $H^0(M)$-invariant.
	\end{proof}

	\begin{lem}\label{kconjg}
		Let $g \in G$ and $\epsilon \in \mu_G$ such that $\tpo{g}{M_0, \epsilon}$ is finitely satisfiable in $M_0$, then $\epsilon^g \in \mu_G$. Particularly, if $N \succ M_0$ and $g \in G$ such that $\tpo{g}{N}$ is finitely satisfiable in $M_0$, then $(\mu_G(N))^g \leq \mu_G$.
	\end{lem}
	
	\begin{proof}
		Suppose that $g$ and $\epsilon$ satisfy the condition, and $\epsilon^g \notin \mu_G$, then there exists an $M_0$-definable open neighorhood $U$ around $\id_G$ satisfies $\epsilon^g \notin U$. Thus the formula $(\epsilon^x \notin U)$ is in $\tpo{g}{M_0, \epsilon}$. So by the type is finitely satisfiable in $M_0$, we know that there exists $g_0 \in G(M_0)$ such that $\epsilon^{g_0} \notin U$, hence is not in $\mu_G$. A contradiction.
	\end{proof}

	\begin{lem}\label{ctp}
		Let  $\epsilon_0, \epsilon \in \mu_G$ and  $b_0, b \in B$ such that $\epsilon_0b_0 = \epsilon b$, then $b_0, b$ are in the same coset of $B^0$.
	\end{lem}
	\begin{proof}
		$bb_0^{-1} = \epsilon^{-1}\epsilon_0 \in \mu_G \cap B=\mu_B$. Since each $\emptyset$-definable  subgroup $A$ of $B$ with finite index is open, we have $\mu_B\models A$, so $\mu_B\leq B^0$, which means $b_0$ and $ b$ are in the same coset of $B^0$.
	\end{proof}

	We will freely use the above fact. Note that for any $b\in B$, and any finite-index subgroup $A$ of $B$ defianble over $M_0$, there is $b_0\in B(M_0)$ such that $bA=b_0A$. So the  coset $bB^0$  is defined by a partial type over $M_0$ for any $b\in B$.

	Let $\ia\sq S_A(M^\ext)$ and $\jh\sq S_H(M^\ext)$ be minimal subflows with $u\in \ia$ and $v\in \jh$ are idempotents. Then by Lemma \ref{IA*JH}, $\jb=\ia*\jh$ is a minimal subflow of $S_B(M^\ext)$ and $p_0=u*v\vdash B^0$ is an idempotent. By Lemma \ref{Ellis-group-in-B}, the Ellis group $\eb=p_0*\jb$ of $S_B(M^\ext)$ equals to $u*\ia*\jh$. 
	
	By Fact \ref{definably-amenable}(iv), $\tau: p\mapsto p/B^0$ is an isomorphic from $\eb$ to $B/B^0$.  Suppose that $\delta \in B$,  consider the map $l_\delta: \eb \to \eb$ defined by $p \mapsto p_\delta*p$, where $p_\delta\in \eb$ such that $p_\delta/B^0=\delta/B^0$. Then $l_\delta$ is a bijection from $\eb$ to itself, and $l_{\delta_1}=l_{\delta_2}$ iff $\delta_1/B^0=\delta_2/B^0$.  Besides, for any $\delta_1, \delta_2 \in B$, we have $l_{\delta_1} \circ l_{\delta_2} = l_{\delta_1\delta_2}$.

	By Lemma \ref{klem}, we see that $p * q \vdash \mu_GB$ for each $p\in \jb$ and $q\in S_K(M^\ext)$. We now assume that $u\in \ia$ and $v\in \jh$ are an idempotents,   $\eb=\ea*\jh$ is the Ellis group of $S_B(M^\ext)$, and   $p_0=u*v\in \eb$ is the idempotent.
	
	\begin{lem}\label{l(p0)}
		If $\delta=ah$ with $a\in A$ and  $h\in H$. Let $q\in\ea$ such that $q/A^0=a/A^0$, $p\in S_H(M^\ext)$ such that $p$ is finitely satisfiable in $M_0$ and $p/N^0=h/H^0$ . Then 
		\[
		l_\delta(p_0)=q*p*p_0.
		\]
	\end{lem}
	
	\begin{proof}
		Assume that  $p_0=u*v$ with $u$ and $v$ be idempotents of $\ea$ and $\jh$ respectively. Let $N^*$ be an $|M|^+$-sarurated extension of $M^\ext$, $a_1\in A(N^*)$ and $h_1\in H(N^*)$ such that $a_1\models u$ and $h_1\models v|(M^\ext, a_1)$. Let $N^{**}$ be an $|N^*|^+$-sarurated extension of $N^*$. Take $a_0\in A$ realizing the generic extension of $q$ over $N^{**}$. Take $h_0\in H(N^{**})$ realizing  $p$ such that   $\tp{h_0/N^*}$ is finitely satisfiable in $M_0$. Then 
		\[
		q*p*p_0=\tp{a_0h_0a_1h_1/M^\ext}=\tp{a_0{a_1}^{h_0}h_0h_1/M^\ext}.
		\]
		Since $\tp{h_0/N^*}$ is finitely satisfiable in $M_0$, we see that \[{a_1}^{h_0}\in (\mu_G(N^*))^{h_0}\cap A(N^{**})\leq \mu_G\cap A(N^{**})=A^0(N^{**})\]
		by Lemma \ref{kconjg}. Now $\tp{a_0/N^{**}}$ is a generic type and thus $A^0(N^{**})$-invariant under the right action. So $\tp{a_0/N^{**}}=\tp{a_0{a_1}^{h_0}/N^{**}}$ is finitely satisfiable in $M_0$. We conclude that
		\begin{align*}
			\tp{a_0{a_1}^{h_0}h_0h_1/M^\ext}&=\tp{a_0{a_1}^{h_0}/M^\ext}*\tp{h_0h_1/M^\ext}\\
			&=\tp{a_0/M^\ext}*\tp{h_0h_1/M^\ext},
		\end{align*}
		which is in $\eb$ By Lemma \ref{Ellis-group-in-B}. Since 
		\[\tp{a_0/M^\ext}*\tp{h_0h_1/M^\ext}\vdash a_0h_0/B^0=ah/B^0=\delta/B=l_\delta(p_0)/B,\]
		we have 
		\[
		q*p*p_0=\tp{a_0/M^\ext}*\tp{h_0h_1/M^\ext}=l_\delta(p_0).
		\]
		
	\end{proof}

	\begin{lem}\label{dgroup}
		If $p = l_{\delta}(p_0)$ and $p^\prime = l_{\delta^\prime}(p_0)$, then $p * p^\prime = l_{\delta\delta^\prime}(p_0)$.
	\end{lem}
	\begin{proof}
		Since $p\vdash \delta/B^0$ and  $p\vdash \delta'/B^0$, we have $p*p'\vdash \delta\delta'/B^0$. On the other side $l_{\delta\delta^\prime}(p_0)\vdash \delta\delta'/B^0$, which implies that  $p*p'=l_{\delta\delta^\prime}(p_0)$ by Fact \ref{definably-amenable}(iv).
	\end{proof}

	\begin{lem}\label{kb}
		Let $q \in S_G(M^\ext)$ such that $p_0 * q \vdash \mu_G\delta_0 B^0$, let $p_1 = l_\delta(p_0)$, where $\delta_0, \delta \in B$. Then $p_1 * q \vdash \mu_G\delta\delta_0B^0 $.
	\end{lem}
	\begin{proof}
		Let $N^{**}\succ N^*\succ M$, where $N^*$ is $|M|^+$-saturated and $N^{**}$ is $|N^*|^+$-saturated. Let $\delta=ah$ with $a\in A$ and $h\in H$.

		Let and $a_0\in A$ such that $\tp{a_0/M^\ext}\in E_A^\ext$ and $\tp{a_0/N^{**}}\vdash a/A^0$ is finitely satisfiable in $M_0$. As $h/H^0$ is a partial type over $M_0$, there is $h_0\in H(N^{**})$ such that $\tp{h_0/N^*}\vdash h/H^0$ is finitely satisfiable in $M_0$. Without loss of generality, we may assume that $\delta,\delta_0\in B^(N^*)$.
		
		By Lemma \ref{l(p0)}, 
		\[
		l_\delta(p_0)=\tp{a_0/M^\ext}*\tp{h_0/M^\ext}*p_0.
		\]
		Let $b\in B(N^*)$ and $g\in G(N^*)$ such that  $b\models p_0$ and $g\models q|(M^\ext,b)$, then there are $c\in \mu_G(N^*)$ and $b'
		\in \delta_0B^0(N^*)$ such that
		$bg=cb'$. Now 
		\[
		l_\delta(p_0)*q=\tp{a_0/M^\ext}*\tp{h_0/M^\ext}*p_0*q=\tp{a_0h_0cb'/M^\ext}.
		\]By Lemma \ref{kconjg}, ${c}^{h_0}\in \mu_G(N^{**})$ and thus
		\begin{align*}
			a_0h_0cb' = a_0{c}^{h_0}h_0b' \in a_0\mu_Gh_0\delta_0B^0=\mu_Ga_0h_0\delta_0B^0=\mu_G\delta\delta_0B^0.
		\end{align*}
		This completes the proof.
	\end{proof}
	
	Take a generic type $q_M\in S_K(M^\ext)$ such that $q_M\vdash K^0$. We fix some $\delta_0\in B$ such that  $p_0*q_M\vdash K^0\delta_0 B^0$. By Lemma \ref{kb} we have:

	\begin{cor}\label{left-action-on-idempotent}
		Write $p_M = l_{\delta_0^{-1}}(p_0)$. Then $l_{\delta}(p_M) * q_M \vdash \mu_GB^0\delta$ for each $\delta\in B$. Particularly, $p_M*q_M\vdash \mu_G B^0=K^0B^0$
	\end{cor}

	We now prove that $q_M * p_M$ is an idempotent in a minimal subflow. We will  denote the space of generic types in $S_K(M^\ext)$ by $\ik$.


	\begin{lem}\label{qpqp}
		Suppose $q_1 \in \ik$, $p, p^\prime \in \eb$. If  $p = l_{\delta}(p_M)$ for some  $\delta \in B$. Then we have $q_1 * p * q_M * p^\prime = q_1 * l_{\delta}(p^\prime)$.
	\end{lem}
	\begin{proof}
		By Corollary \ref{left-action-on-idempotent}, $p * q_M\vdash\mu_G\delta B^0$. Without loss of generality, we main assume that 
		\[
		q_1* p * q_M * p^\prime=q_1*\tp{\epsilon \delta/M^\ext} * p^\prime
		\]
		for some $\epsilon \in \mu_G$. It is easy to see that $q_1*\tp{\epsilon \delta/M^\ext}=q_1*\tp{\delta/M^\ext}$ since $q_1$ is generic. 
		
		Let $u\in\ea$, then by Lemma \ref{Ellis-gp-in-B-II}, $u*S_B(M^\ext)*p'\sq \eb$. We assume that $u\vdash A^0\sq \mu_G$ is an idempotent. So $q_1*u=q_1$. We have
		\[
		q_1* p * q_M * p^\prime=q_1*\tp{\delta/M^\ext}*p'=q_1*(u* \tp{\delta/M^\ext}* p^\prime).
		\]
		Assume that $p'\vdash \delta'B^0$, then 
		we have  $u* \tp{\delta/M^\ext}* p^\prime\vdash \mu_G \delta\delta'B^0$. Since   both $u* \tp{\delta/M^\ext} * p^\prime$ and $l_\delta(p')$  are in $\eb$, and 
		\[
		l_\delta(p')/B^0= (\delta\delta')/B^0=(u* \tp{\delta/M^\ext} * p^\prime)/B^0,
		\]
		we conclude that $l_\delta(p')=u* \tp{\delta/M^\ext} * p^\prime$ by Fact \ref{definably-amenable}. Thus
		\[
		q_1* p * q_M * p^\prime=q_1*(u* \tp{\delta/M^\ext} * p^\prime)=q_1*l_\delta(p')
		\]
		as required.
	\end{proof}

	
	By Lemma \ref{qpqp}, we see immediately that
	\begin{cor}
		$q_M * p_M  * q_M * p_M = q_M * p_M$. Namely, $q_M * p_M$ is an idempotent
	\end{cor}
	
	We now show  that the subflow  $S_G(M^\ext) * q_M * p_M$  generated by  $q_M * p_M$ is minimal. We use $KA$ to denote the definable set $\{ka|\ k\in K,\ a\in A\}$.

	\begin{lem}\label{kj}
		Suppose that $p\in\jh$. Then 
		\[
		S_G(M^\ext)*p\sq S_{KA}(M^\ext) *\jh.
		\]
	\end{lem}
	\begin{proof}
		Let $N^*$ be an $|M|^+$-saturated extension of $M^\ext$. Let $s \in S_G(M^\ext)$. Take $k_0\in K(N^*)$,  $b_0\in B(N^*)$ and $b\in B$ such that $k_0b_0 \models s$  and $h \models p  | N^*$. Then:
		$$s * p  = \tpo{k_0b_0h}{M^\ext}.$$
		Suppose that $b_0 = a_0h_0$ with $a_0\in A(N^*)$ and $h_0\in H(N^*)$, then $s * p$ is $\tpo{k_0a_0h_0h}{M^\ext}$.
		
		By Fact \ref{B(Mext)-has-dfg}, the heir $p| N^*$ is an $f$-generic type in $S_B(N^*)$ and any $B(N^*)$-translate of $p| N^*$ is $f$-generic and definable over $M$, so there is $q^*\in \jh$ such that $\tpo{h_0h}{N^*}=q^*|N^*$. Therefore we have:
		\begin{align*}
			s * p  &= \tpo{k_0a_0h_0h}{M^\ext}\\
			&=\tpo{k_0a_0}{M^\ext}*\tpo{h_0h}{M^\ext}\\
			&\in S_{KA}(M^\ext)*S_H(M^\ext).
		\end{align*}
		Therefore $S_G(M^\ext) * p \subseteq S_{KA}(M^\ext) * \jh$.
	\end{proof}

	\begin{pro}\label{AP}
		Suppose $s \in S_G(M^\ext)$, then 
		\[
		q_M * p_M \in S_G(M^\ext) * s * q_M * p_M=\cl{G(M)\cdot (s * q_M * p_M)}.
		\]
		Consequently, $S_G(M^\ext) * q_M * p_M$ is a minimal subflow.
	\end{pro}
	\begin{proof}
		By the previous lemma, we may assume that $s * q_M * p_M = q * p_1$ where $p_1 = \jh$ and $q \in S_{KA}(M^\ext)$. Let $N^*$ be an $|M|^+$ saturated extension of $M^\ext$. Let $k_0\in K(N^*)$ and $a_0\in A(N^*)$ such that $k_0a_0\models q$.
		Let $u\in \ea$ be an idempotent. Then $u*\tp{a_0/M^\ext}*p_1\in \ea*\jh=\eb$.

		If $\tp{a_0/M^\ext}*p_1\vdash \delta/B^0$, then by Lemma \ref{qpqp} we see  that 
		\[
		q_M*l_{\delta^{-1}}(p_M)*q_M*(u*\tp{a_0/M^\ext}*p_1)=q_M*p_M.
		\]
		Note that $q_M*u=q_M$ since $u\vdash A^0\sq K^0=\mu_G$. So we have
		\begin{align*}
			q_M*p_M &= q_M*l_{\delta^{-1}}(p_M)*q_M*(u*\tp{a_0/M^\ext}*p_1) \\
			&=q_M*l_{\delta^{-1}}(p_M)*q_M*\tp{a_0/M^\ext}*p_1.
		\end{align*}
		By Lemma \ref{fsgms}(iii), there is $r\in q_M*\ik$ such that $q_M=r*\tp{k_0/M^\ext}$. So we have 
		\[
		q_M*\tp{a_0/M^\ext}*p_1=r*\tp{k_0/M^\ext}*\tp{a_0/M^\ext}*p_1.
		\]
		Take $k\in K(N^{*})$ realizing the unique generic (or coheir) extension of  $r$ over $(M^\ext, k_0,h_0)$ and $h\in H$ realizing the heir of $p_1$ over $N^{*}$. Then $\tp{h_0h/N^*}$ the unique heir of some $p\in \jh$. We see that 
		\begin{align*}
			&r*\tp{k_0/M^\ext}*\tp{a_0/M^\ext}*p_1\\
			&=(\tp{k/M^\ext}*\tp{k_0/M^\ext})*(\tp{h_0/M^\ext}*\tp{h/M^\ext})\\
			&=\tp{kk_0/M^\ext}*\tp{h_0h/M^\ext}=\tp{kk_0h_0h/M^\ext}.
		\end{align*}
		On the other side, 
		\[
		r*q*p_1=r*\tp{k_0a_0/M^\ext}*p_1=\tp{kk_0h_0h/M^\ext}.
		\]
		We conclude that
		\begin{align*}
			q_M*p_M & = q_M*l_{\delta^{-1}}(p_M)*r*\tp{k_0/M^\ext}*\tp{a_0/M^\ext}*p_1\\
			&= q_M*l_{\delta^{-1}}(p_M)*r*q*p_1\\
			&=q_M*l_{\delta^{-1}}(p_M)*r*s * q_M * p_M,
		\end{align*}
		which is in $ S_G(M^\ext) * s * q_M * p_M$ as required.
	\end{proof}

	Let $q_M$, $p_M$ be as in above, and  $\m=S_G(M^\ext) * q_M * p_M$ be the minimal subflow generated by $q_M * p_M$,  we  now compute the Ellis group $E$ of $S_G(M^\ext)$, which is of the form $q_M * p_M*\m$.
	
	\begin{lem} Assume again that $\eb$ is the Ellis group of $S_B(M^\ext)$ generated by $p_M$. Then
		\[
		q_M * p_M * \m = q_M * \eb
		\]
	\end{lem}
	\begin{proof}
		Let $p_1 \in \eb$. By Lemma \ref{qpqp},
		\begin{align*}
			q_M *p_1&= q_M*p_M*q_M * p_1 \\
			&= q_M*p_M*q_M*p_M*q_M * p_1\\
			&=q_M*p_M*(q_M * p_1*q_M*p_M)\in q_M * p_M * \m.
		\end{align*}
		Thus, we have  $q_M * \eb \subseteq q_M * p_M * \m$.
		
		Now we show that $q_M * p_M * \m \subseteq q_M * \eb$. 
		By lemma \ref{kj}, $\m$ is a subset of $S_{KA}(M^\ext)*\jh$, so it suffices to show $q_M * p_M * S_{KA}(M^\ext) * \jh \subseteq q_M * \eb$.

		Let $q \in S_{KA}(M^\ext)$ and $p_1 \in \jh$. By Lemma \ref{klem}, $p_M*q\vdash \mu_GB$. Let $N^*\succ M^\ext$ be $|M|^+$-saturated. Assume that $p_M*q=\tp{\epsilon a_0 h_0/M^\ext}$ for $\epsilon\in \mu_G(N^*)=K^0(N^*)$, $a_0\in A(N^*)$, and $h_0\in H(N^*)$.

		Let $k\in K(N^*)$ realize the  coheir of $q_M$ over $\dcl{M^\ext,\epsilon, a_0,h_0}$ and $h\models p_1|N^*$. Then a similar argument as in Proposition \ref{AP} shows that 
		\begin{align*}
			q_M * p_M *q*p_1&=\tp{k\epsilon a_0 h_0 h/M^\ext}\\
			&=\tp{k\epsilon /M^\ext}*\tp{a_0/M^\ext}*\tp{h_0 h/M^\ext}\\
			&=q_M*\tp{a_0/M^\ext}*\tp{h_0 h/M^\ext}.
		\end{align*}
		Let $u\in \ea$ be the idempotent. Then $q_M*u=q_M$, so we have
		\begin{align*}
			q_M*\tp{a_0/M^\ext}*\tp{h_0 h/M^\ext}&=q_M*u*\tp{a_0/M^\ext}*\tp{h_0 h/M^\ext}\\
			&\in q_M*\ea*\jh=q_M*\eb.
		\end{align*}
		This completes the proof.
	\end{proof}

	\begin{thm}
		The Ellis group of $S_G(M^\ext)$ is isomorphic to Ellis group of $S_B(M^\ext)$. Namely,  $q_M * \eb\cong\eb$.
	\end{thm}
	\begin{proof}
		Recall that $p_M\vdash B^0 \delta_0^{-1}$. Let $r: \eb \to q_M * \eb$ be the map defined by $p \mapsto q_M * l_{\delta_0^{-1}}(p)$. Because $l_{\delta_0^{-1}}$ is a bijection, we know $r$ is an onto map.\par
		Suppose $p, p^\prime \in \eb$ and assume $p = l_{\delta}(p_M)$, $p^\prime = l_{\delta^\prime}(p_M)$, then by lemma \ref{qpqp}:
		\begin{align*}
			r(p) * r(p^\prime) = q_M * l_{\delta_0^{-1}\delta}(p_M) * q_M * l_{\delta_0^{-1}\delta^\prime}(p_M) 
			= q_M * l_{\delta_0^{-2}\delta\delta^\prime}(p_M) .
		\end{align*}
		Now $l_{\delta_0^{-2}\delta\delta^\prime}(p_M)\vdash \delta_0^{-3}\delta\delta^\prime B^0$
		and $p * p^\prime\vdash \delta_0^{-2}\delta\delta^\prime B^0$, so \[
		l_{\delta_0^{-2}\delta\delta^\prime}(p_M)=l_{\delta_0^{-1}}(p*p'),\]
		and which implies that 
		\[
		r(p) * r(p^\prime)=q_M * l_{\delta_0^{-2}\delta\delta^\prime}(p_M)=q_M*l_{\delta_0^{-1}}(p*p')=r(p*p').
		\]
		So  $r$ is a group homomorphism.
		
		Now it remains to show that $r$ is injective. For $p = l_{\delta}(p_M)$, we have $r(p) = q_M * l_{\delta_0^{-1}\delta}(p_M) \vdash \mu_G B^0\delta_0^{-2}\delta$. If $r(p) = q_M * p_M \vdash \mu_G B^0\delta_0^{-1}$, then there are $\epsilon \in\mu_G$ and $b_1,b_2\in B^0$ such that $\epsilon b_1 \delta_0^{-2}\delta=b_1 \delta_0^{-1}$. We see that $\epsilon\in \mu_G\cap B\sq B^0$,  and conclude immediately that $\delta/B^0=\delta_0/B^0$, thus $p\vdash B^0$ is the idempotent. So $\ker(r)= \{\id_{\eb}\}$, i.e. $r$ is an isompophism.
	\end{proof}
	Thus finally we have our main theorem:
	\begin{thm}\label{main}
		Suppose that $G$ is a linear algebraic group over $\Q$ admits a Iwasawa decomposition $KB$, with $K$ open and definably compact over $\Q$ and $B$ a borel subgroup deinable over $\Q$. Then the Ellis group of $S_G(M^\ext)$ algebraically isomorphic to $ B/B^0$.
	\end{thm}
	
	\begin{Example}
		We now consider the general linear group $G(x)=\GL(n,x)$. Then $G(\Q)$ has the Iwasawa decomposition $G(\Q)=K(\Q)B(\Q)$, where $K(\Q)=\GL(n,\Z)$ is a maxiaml open compact subgroup, and $B(\Q)$ is the subgroup consisting of all upper triangular matrices. Since $B=D\rtimes B_u$, where $D$ is the subgroup of diagonal matrices and $B_u$ is the subgroup of strictly upper triangular matrices, we see that $B/B^0\cong D/D^0\cong (\Gm/\Gm^0)^n$. Now $\Gm/\Gm^0$ is isomorphic to $(\hat{\z}\times \Z^*)$, with $\hat{\z}=\invlim{\z/n}$ and $\Z^*=\{x\in \Z|\ \nu(x)=0\}$ by Remark 2.5 in \cite{PPY-sl2qp}. We finally conclude that the Ellis group corresponding to $\GL(n,x)$ is isomorphic to $(\hat{\z}\times \Z^*)^n$, independent of the models. 
		
		Similarly,  for $G(x)=\SL(n,x)$,  we have that the corresponding Ellis group is isomorphic to $(\hat{\mathbb Z}\times \Z^*)^{n-1}$. 
	\end{Example}

	\vspace*{2ex}  
	Acknowledgement:\ The research is supported by The National Social Science Fund of China(Grant No.20CZX050).

\end{document}